\documentclass[12pt]{article}
\usepackage{a4}

\usepackage{amsthm}
\usepackage{amsmath}
\usepackage{enumerate}
\usepackage{graphicx}

\newtheorem{theorem}{Theorem}
\newtheorem{prop}[theorem]{Proposition}
\newtheorem{lemma}[theorem]{Lemma}
\newtheorem{coro}[theorem]{Corollary}
\newtheorem{conj}[theorem]{Conjecture}

\newcommand\size[1] {\left|{#1}\right|}

\newcommand\Setx[1] {\left\{{#1}\right\}}
\newcommand\map[3] {{#1}:\,{#2}\to{#3}}
\newcommand{\sm}{\setminus}
\newcommand{\AAA}{\mathcal A}
\newcommand{\BB}{\mathcal B}

\newcommand{\FF}{\mathcal F}

\title{\textbf{10-tough chordal graphs are
    Hamiltonian}\footnote{Research supported by project GA14-19503S of
    the Czech Science Foundation. The work of the first author is
    supported by project LO1506 of the Czech Ministry of Education,
    Youth and Sports.}}%
\author{Adam Kabela\thanks{Department of Mathematics and European
    Centre of Excellence NTIS (New Technologies for the Information
    Society), University of West Bohemia, Pilsen, Czech
    Republic. Email: \texttt{kabela@kma.zcu.cz}.}%
  \and Tom\'{a}\v{s} Kaiser\thanks{Department of Mathematics,
    Institute for Theoretical Computer Science (CE-ITI), and European
    Centre of Excellence NTIS (New Technologies for the Information
    Society), University of West Bohemia, Pilsen, Czech
    Republic. Email: \texttt{kaisert@kma.zcu.cz}.}}
    
\date{}

\begin{document}
\maketitle

\begin{abstract}
  Chen et al. proved that every $18$-tough chordal graph has a
  Hamilton cycle [Networks 31 (1998), 29-38]. Improving upon their
  bound, we show that every $10$-tough chordal graph is Hamiltonian
  (in fact, Hamilton-connected). We use Aharoni and Haxell's
  hypergraph extension of Hall's Theorem as our main tool.
\end{abstract}


\section{Introduction}
\label{sec:intro}

We study Hamilton cycles and toughness in chordal graphs. Recall
that following Chv\'{a}tal~\cite{chva}, the \emph{toughness} of a
graph $G$ is the minimum, taken over all separating sets $S$ of
vertices of $G$, of the ratio of $|S|$ to the number of components of
$G - S$. If $G$ is complete, the toughness is defined to be $\infty$. We
say that a graph is \emph{$t$-tough} if its toughness is at least
$t$. It is easy to observe that Hamiltonian graphs are $1$-tough. In the
reverse direction, Chv\'{a}tal~\cite{chva} conjectured the following:

\begin{conj}\label{conj:ch}
  There exists $t_0$ such that every $t_0$-tough graph
  (on at least $3$ vertices) is Hamiltonian.
\end{conj}

Conjecture~\ref{conj:ch} is still open. The best available lower bound
is due to Bauer, Broersma and Veldman~\cite {74} who constructed
non-Hamiltonian graphs with toughness arbitrarily close to $\tfrac94$.  

Partial results related to Chv\'{a}tal's conjecture have been obtained
in various restricted classes of graphs (see the survey~\cite{surv}
for details). A number of these results concern chordal graphs.
For instance, it is known that 
(with the exception of $K_1$ and $K_2$) every
chordal planar graph of toughness
more than $1$ is Hamiltonian~\cite{chpl}, and so is every $1$-tough
interval graph~\cite{inte1} or every $\frac{3}{2}$-tough split
graph~\cite{spli}. All of these results are tight.

Non-Hamiltonian chordal graphs with toughness arbitrarily close to
$\tfrac74$ were constructed in~\cite{74}. On the other hand,
Chen et al.~\cite{18} showed that every $18$-tough chordal graph
is Hamiltonian. In this paper, we improve the bound as follows:

\begin{theorem}\label{t:main}
  Every $10$-tough chordal graph on at least $3$ vertices is Hamiltonian.
\end{theorem}

The construction of the Hamilton cycle is based on auxiliary graphs
that are defined in Section~\ref{sec:overspan-graphs} to encode the
local structure of a given chordal graph.  Our main tool is a
hypergraph extension of Hall's Theorem, due to Aharoni and
Haxell~\cite{repr}; its application is described in
Section~\ref{sec:hall}. The proof of Theorem~\ref{t:main} is given in
Section~\ref{sec:tough}. We conclude the paper in
Section~\ref{sec:hc} discussing a strengthening of
Theorem~\ref{t:main} to Hamilton-connectedness.


\section{Tree representations and overspan graphs}
\label{sec:overspan-graphs}

For a graph $H$, let $V(H)$ denote the set of vertices, $E(H)$ the set of edges,
$c(H)$ the number of components of $H$. 
By a well-known theorem of Gavril~\cite{gavr}, for every chordal graph $G$
there exists a \emph{tree representation} of $G$
--- that is, a tree $T_0$ and a family $\FF$ of subtrees of $T_0$ such
that $G$ is isomorphic to the intersection graph of $\FF$.
For each vertex $v$ of $G$, let $F_v$ denote the corresponding subtree in $\FF$.

For a given chordal graph $G$, we choose a tree representation
$(T_0, \FF)$ such that the tree $T_0$ has minimal number of
vertices. Thus, for each leaf of $T_0$, there is a 
subtree in $\FF$ consisting of the leaf as its only vertex.
We fix this tree representation and choose an independent set $I$ in $G$
that is maximal with the property that for each $v\in I$, $F_v$
is a path all of whose vertices have degree at most $2$ in $T_0$.
Moreover, we choose $I$ such that for every $v\in I$, $F_v$ contains no subtree
of $\FF$ as a proper subgraph. For $v\in I$, a path $F_v$ is called an
\emph{$I$-path}; it is \emph{trivial} if it consists of a single vertex.
To emphasize the
distinction between the edges contained in $I$-paths and the other
edges, we colour each edge of $T_0$ either \emph{red} (if it belongs
to some $I$-path) or \emph{black} (otherwise).

Next, we fix the choice of the independent set $I$ and we modify $T_0$
into a tree $T$ which we call the base tree for $G$.
One by one, we suppress each degree $2$ vertex of $T_0$ that is not an
endvertex of any $I$-path (a trivial $I$-path has one endvertex).
The resulting tree $T$ (the
\emph{base tree} for $G$) inherits a red-black colouring of edges.
We observe that nontrivial $I$-paths in $T_0$ correspond one-to-one to red
edges in $T$, furthermore the red edges form a matching and their
endvertices are all of
degree $2$. Vertices of $T_0$ that exist also in $T$ are
called \emph{substantial} (that is, substantial vertices are the
endvertices of $I$-paths and vertices of degree at least $3$).
For further reference, let us state the following observation:
\begin{prop}\label{p:subst}
  For every vertex $v$ of $G$, the tree $F_v$ contains a substantial vertex.
\end{prop}
\begin{proof}
To the contrary, suppose there is a vertex $v$ such that the tree $F_v$ contains
no substantial vertex. That is, $F_v$ contains
neither a vertex whose degree in $T_0$ is at least $3$ nor an endvertex of any $I$-path.
In particular, $v \not \in I$ and by the choice of $I$, $F_v$ is not a proper subgraph of any $I$-path.
Hence $F_v$ does not intersect any $I$-path, so $v$ is not adjacent to any 
vertex of $I$. We obtain a contradiction with the maximality of $I$.
\end{proof}

\begin{figure}[h]
    \centering
    \includegraphics[scale=0.8]{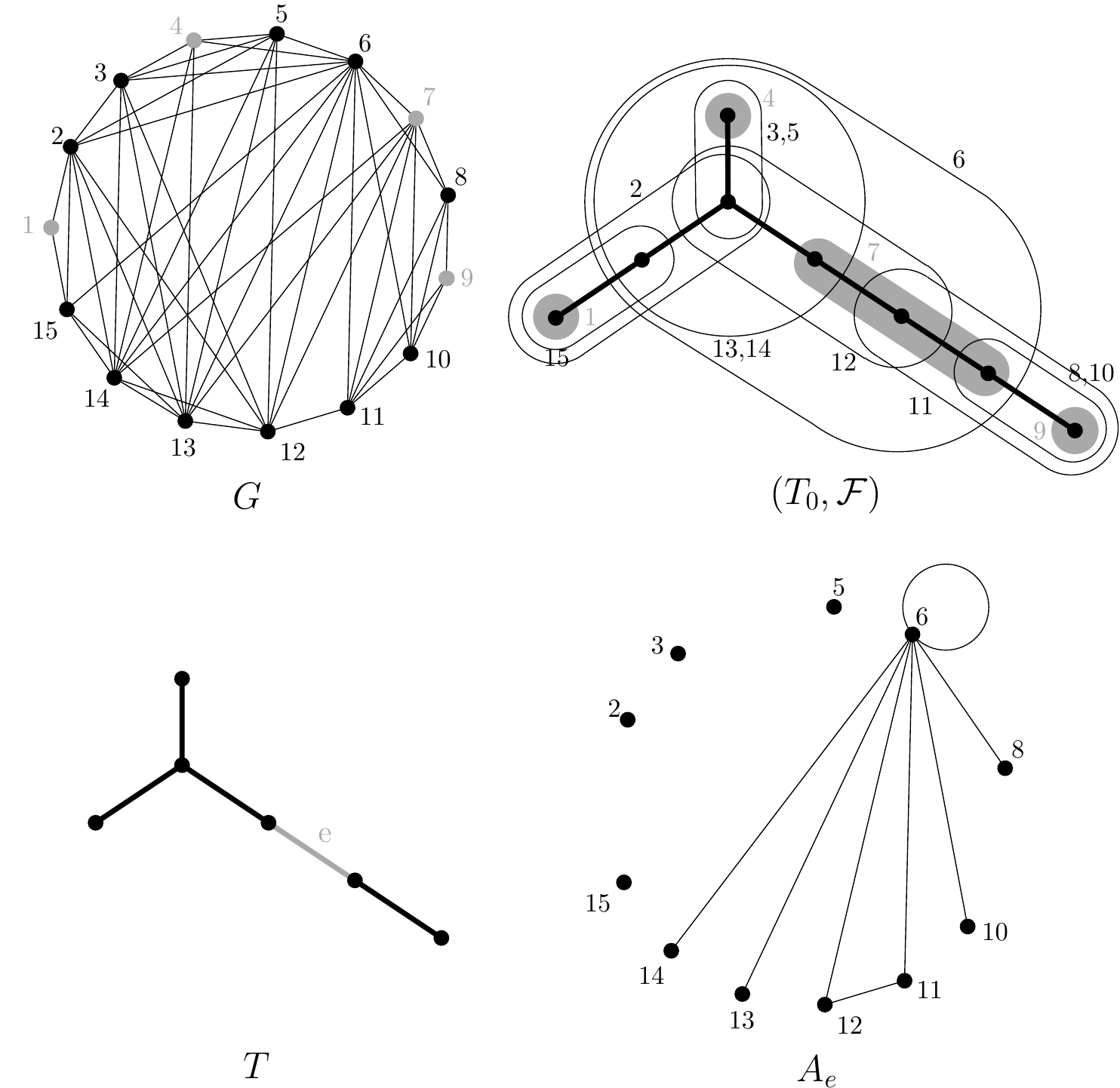}
    \caption{A chordal graph $G$, a tree representation $(T_0, \FF )$,
    a base tree $T$ and an overspan graph $A_e$ assigned to a
    red edge $e$. In the tree representation $(T_0, \FF )$ (top-right),
    the ovals depict subtrees of the tree $T_0$ that belong to $\FF$.
    The subtrees of $T_0$ and vertices of $A_e$ 
    are indexed by the same integer as the corresponding 
    vertices of $G$. In grey, we highlight the vertices of the set $I$ in $G$
    (top-left), the $I$-paths in $T_0$ (top-right) and the red edge $e$ in $T$
    (bottom-left).}
\end{figure}

We use $T$ to construct a family of so-called \emph{overspan graphs},
assigning one such graph $A_e$ to each edge $e$ of $T$. The vertex set
of $A_e$ is $V(G)\sm I$. The graph $A_e$ may contain loops; to avoid
ambiguity, we point out that we view a loop as an edge of a special
type. To describe the edges of $A_e$, we let $r$ and $s$ denote the
endvertices of $e$. (Note that these are substantial vertices of
$T_0$.)  The edge set of $A_e$ is defined as follows:
\begin{itemize}
\item there is a loop on a vertex $v$ if $F_v$ contains the vertices
  $r$ and $s$ in $T_0$,
\item vertices $u$ and $v$ are connected by an edge if $r\in V(F_u)$ and
  $s\in V(F_v)$ (or vice versa), and $uv$ is an edge of $G$.
\end{itemize}
Furthermore, for each black edge $e$ of $T$ we assign to $e$ an additional
overspan graph which is a copy of $A_e$. 

The family of overspan graphs for $G$ is constructed for a particular
tree representation $(T_0,\FF)$ and an independent set $I$. As the tree
representation and the independent set are fixed, let us use the
notation $\AAA(G)$ for the family of overspan graphs.

For $\BB \subseteq \AAA(G)$, we define a graph $G_\BB$ on vertex set
$V(G)\sm I$. The edge set of $G_\BB$ is the union of the edge sets of 
all the graphs that belong to $\BB$; each edge is included at most once
in this union. In case $\BB = \AAA(G)$, we let the graph be
denoted $G_\AAA$. 

The reason for the name `overspan graph' is that we view each edge of
$T$ as representing a gap that needs to be crossed by the desired
Hamilton cycle, and the edges of the corresponding overspan graph
encode the possible ways of doing so.
We conclude this section by pointing out a connection between the
family $\AAA(G)$ and the Hamiltonicity of $G$. In graphs with
loops (such as the overspan graphs and their unions), we allow loops
in matchings, as long as they are vertex-disjoint from
the other edges of the matching.

\begin{lemma}\label{l:ham}
  Let $G$ be a chordal graph on at least $3$ vertices and let $\AAA(G)$
  be the family of overspan graphs for $G$ (with respect to a tree
  representation of $G$ and an independent set $I$). Assume that we can
  choose one edge from each graph in $\AAA(G)$ in such a way that the
  chosen edges form a matching in $G_\AAA$. Then $G$ is Hamiltonian.
\end{lemma}

\begin{proof}
Let $M$ be the set of chosen edges that form a matching in $G_{\AAA}$.
We assume $T$ has $m$ edges ($m \geq 1$), and we fix an Euler tour $e_0, e_1, ..., e_{2m-1}$ in the
symmetric orientation of $T$. With every directed edge $e_i = t_i t_i'$ of the tour,
we associate a pair of subtrees $(F_i, F_i')$ of $\FF$ as follows.

For every edge $e$ of $T$ there are two corresponding directed
edges, say $e_i$ and $e_j$, in the symmetric orientation.
We discuss two cases: either $e$ is black or it is red.
If $e$ is black, then there are two assigned graphs in $\AAA(G)$, say $A_{e_i}$ and $A_{e_j}$.
By the assumption of the lemma, for $A_{e_i}$ there is a chosen edge of $M$,
namely a simple edge $uv$ or a loop $v$,
and we consider a pair of subtrees $(F_u, F_v)$ or $(F_v, F_v)$ of $\FF$ 
(recall the definition of edges of overspan graphs).
We associate $e_i$ with this pair of subtrees of $\FF$ and
associate $e_j$ with the pair of subtrees of $\FF$ obtained analogously for $A_{e_j}$.

Similarly, if $e$ is red, then there is one assigned graph in $\AAA(G)$
and a chosen edge of $M$, which gives a pair of subtrees of $\FF$ and we associate $e_i$ with this pair.
To find the associated pair for $e_j$, we recall that a red edge of $T$ is obtained from
a non-trivial $I$-path in $T_0$. We let $F_v$ denote this non-trivial $I$-path related to $e$
and we associate $e_j$ with the pair $(F_v, F_v)$.

We observe that no subtree of $\FF$ is used in more than one associated pair,
considering that the edges of $M$ form a matching in $G_{\AAA}$ and vertices of
$I$ are not included in $G_{\AAA}$.

We traverse the Euler tour $e_0, e_1, ..., e_{2m-1}$ edge by edge, and as we go we build a 
sequence of subtrees of $\FF$ as follows.
When traversing the edge $e_i  = t_i t_i'$ of the tour, we extend the sequence by adding
subtrees of the associated pair $(F_i, F_i')$.
In particular, we add the subtrees in the order $F_i, F_i'$
such that $t_i \in V(F_i)$ and $t_i' \in V(F_i')$.
We obtain a sequence $S = F_0, F_0', F_1, F_1',  ..., F_{2m-1}, F_{2m-1}'$.
By the definition of $S$, every two consecutive subtrees have a vertex in common
(the first and last subtrees are also considered consecutive), and we shall preserve this property even
as we further modify the sequence. 

Now, we extend the sequence so as to include all subtrees of $\FF$.
For every subtree of $\FF$ that is not in $S$, we choose one of its substantial vertices
arbitrarily; and we call it the distinguished vertex of this subtree.
(This is possible due to Proposition~\ref{p:subst}.)
For every vertex $t$ of $T$ in sequence,
we consider an edge of the tour 
incident with $t$ and directed towards $t$, say $e_i$,
and we note that $t \in F_i'$ and $t \in F_{i+1 \pmod{2m}}$.
We extend the sequence by adding all subtrees with a
distinguished vertex $t$ as successors of $F_i'$
in an arbitrary order.

Finally, we remove duplicities from the extended sequence.
For every associated pair of subtrees $(F_v, F_v)$ of $\FF$ that was obtained either
using a loop in $M$ or using a non-trivial $I$-path, we remove one copy of $F_v$ from the extended sequence.
In the resulting sequence, every subtree of $\FF$ occurs exactly once
and every two consecutive subtrees have a vertex in common. 
The assumption $m \geq 1$ implies $|\mathcal{F}| \geq 3$, so the sequence
of the corresponding vertices of $G$ defines a Hamilton cycle.

To complete the proof, we observe that if $T$ has no edge, 
then $G$ is Hamiltonian since it is a complete graph.
\end{proof}


\section{Hall's theorem for hypergraphs}
\label{sec:hall}

In this section, we recall an extension of Hall's Theorem to
hypergraphs due to Aharoni and Haxell~\cite{repr}. We use this result
as a tool to verify the condition in Lemma~\ref{l:ham}.

In accordance with~\cite{repr}, we define a hypergraph as a set
of subsets of a ground set. (In particular, multiple hyperedges are not allowed.)
Let $\AAA = \Setx{H_1, H_2, \dots, H_m}$ be a family of hypergraphs.
A \emph{system of disjoint representatives} for $\AAA$ is a function
$\map f \AAA {\bigcup_{i=1}^m H_i}$ such that for all distinct
$i,j\in\Setx{1,\dots,m}$, $f(H_i)$ is a hyperedge of $H_i$ and $f(H_i)
\cap f(H_j) = \emptyset$.
For $\BB \subseteq \AAA$, let $\bigcup\BB$ denote a hypergraph obtained
as a union of hypergraphs in $\BB$; each hyperedge is included at most once
in this union.
Recall that a \emph{matching} in a hypergraph is a collection
of pairwise disjoint hyperedges.
A corollary of the main result of \cite{repr} is stated here as 
the following theorem. 

\begin{theorem}\label{t:hall}%
  Let $\AAA$ be a family of $n$-uniform hypergraphs. A sufficient
  condition for the existence of a system of disjoint representatives
  for $\AAA$ is that for every $\BB \subseteq \AAA$, there exists a
  matching in $\bigcup\BB$ of size greater than $n(|\BB|-1)$.
\end{theorem}

The nontrivial direction of Hall's Theorem for graphs follows directly
from the $n=1$ case of Theorem~\ref{t:hall}. In the argument, we shall use
the next case, $n=2$, where the members of $\AAA$ are graphs. Indeed,
we intend to apply Theorem~\ref{t:hall} to the family of overspan graphs
$\AAA(G)$, which we regard as hypergraphs with
hyperedges of size $1$ (loops) and $2$ (non-loops).
Recall that the \emph{rank} of a hypergraph is the maximum size of its
hyperedge.  
Theorem~\ref{t:hall} easily extends to non-uniform hypergraphs as follows:

\begin{coro} \label{c:hall}  
  Let $\AAA$ be a family of hypergraphs of rank at most $n$.
  If for every $\BB \subseteq \AAA$, there
  exists a matching in $\bigcup\BB$ of size greater than $n(|\BB|-1)$, 
  then there exists a system of disjoint representatives for $\AAA$.
\end{coro}

\begin{proof}
For every hypergraph $H \in \AAA$, we define an $n$-uniform
hypergraph $H^+$ by adding $n - k$ new vertices for every
hyperedge of size $k$ and extending it to size $n$. 
We let $\AAA^+$ denote the resulting family of hypergraphs,
and for a subfamily $\BB \subseteq \AAA$ we let
$\BB^+$ denote the corresponding subfamily of $\AAA^+$.

By the natural correspondence of hyperedges, $\bigcup\BB^+$ contains
a matching of size greater than $n(|\BB|-1)$, for every $\BB^+ \subseteq \AAA^+$.
Since $\bigcup \BB^+$ is an $n$-uniform hypergraph, by Theorem~\ref{t:hall}
there is a system of disjoint representatives for $\AAA^+$,
and hence also for $\AAA$.
\end{proof}

Recall that the \emph{matching number} $\nu(H)$ is the size of a maximum matching
in graph $H$. The following is a reformulation of Lemma~\ref{l:ham}:
 
\begin{lemma} \label{l:match ham}  
  Let $G$ be a chordal graph on at least $3$ vertices. If for every $\BB
  \subseteq \AAA(G)$, the matching number of $G_\BB$ is greater than
  $2\size\BB-2$, then $G$ is Hamiltonian.
\end{lemma}

\begin{proof}
We view $G_\BB$ as a hypergraph of rank at most $2$.
For any $\BB \subseteq \AAA(G)$, in fact $G_\BB$ is the same
hypergraph as $\bigcup\BB$.  
By Corollary~\ref{c:hall} there exists a system of disjoint
representatives for $\AAA(G)$. The edges in the system form
a matching in $G_\AAA$. By Lemma~\ref{l:ham}, the graph $G$
is Hamiltonian. 
\end{proof}


\section{Vertex covers of the overspan graphs and\\toughness}
\label{sec:tough}

Throughout this section, $G$ is a chordal graph, $(T_0,\FF)$
is a tree representation and $I$ is an independent set
used for the construction of a base tree $T$, $\AAA$ is an associated 
family of overspan graphs and $G_{\BB}$ is the union of graphs in
$\BB \subseteq \AAA$, all defined as in Section~\ref{sec:overspan-graphs}.
In addition, we say an edge $e$ of $T$ is a \emph{$\BB$-edge} 
if the overspan graph assigned to $e$ belongs to $\BB$. 

We concluded Section~\ref{sec:hall} with Lemma~\ref{l:match ham}
that provides a sufficient condition for the Hamiltonicity of $G$ in terms
of the matching numbers $\nu(G_\BB)$. As a next step, we relate
the matching number of $G_\BB$ to its vertex cover number.
Recall that the \emph{vertex cover} of a graph $H$ is a set of its vertices
such that every edge of $H$ is incident with a vertex in this set.
The \emph{vertex cover number} $\tau(H)$ is the size of a minimum
vertex cover of $H$. By the classical theorem of
K\"{o}nig, ${\nu(H)=\tau(H)}$ for every bipartite graph $H$. 
We show that the same equality holds for $G_\BB$.

\begin{lemma}\label{l:konig}
  The graph $G_\BB$ satisfies $\nu(G_\BB) = \tau(G_\BB)$.
\end{lemma}

\begin{proof}
We remove from $G_{\BB}$ all vertices incident with a loop,
and let $G_{\BB}^*$ denote the resulting graph.

First, we show that $G_{\BB}^*$ is bipartite.
We let $B$ denote the set of all $\BB$-edges of $T$.
By definition, a vertex $u$ of $G_{\BB}^*$ is also
a vertex of $G$ and there is a related subtree $F_u$ in the tree 
representation. By Proposition \ref{p:subst},
$F_u$ contains a substantial vertex.
Furthermore, $u$ is not incident with a loop in $G_{\BB}$,
so $F_u$ does not contain both endvertices of any edge of $B$.
Hence $F_u$ contains substantial vertices from just one component of $T - B$.

In $T$, we contract every edge that is not in $B$ and we let $T'$ denote the
resulting tree. Vertices of $T'$ correspond one-to-one to components
of $T - B$. For every vertex $u$ of $G_{\BB}^*$, we
associate $u$ with a vertex of $T'$ such that the corresponding
component of $T - B$ contains all substantial vertices of $F_u$.

Let $u$ and $v$ be vertices adjacent in $G_{\BB}^*$.
By the definition of $G_{\BB}$, there is an edge $xy$ in $B$
such that $x \in V(F_u)$ and $y \in V(F_v)$. 
The vertex of $T'$ associated with $u$ (with $v$) is obtained by
contracting all edges of the component of $T - B$ containing $x$
(containing $y$, respectively).
As $x$ and $y$ are adjacent in $T$, the associated vertices are adjacent in $T'$.
The association of vertices is a graph homomorphism from 
$G_{\BB}^*$ to a tree, thus $G_{\BB}^*$ is bipartite.

Since $\nu(H) \leq \tau(H)$ holds for every graph $H$,
it suffices to prove $\nu(G_{\BB}) \geq \tau(G_{\BB})$ for the graph $G_{\BB}$. 
By K\"{o}nig's theorem, $\nu(G_{\BB}^*) = \tau(G_{\BB}^*)$
since $G_{\BB}^*$ is bipartite.
A matching in $G_{\BB}^*$ extended with all the loops forms
a matching in $G_{\BB}$. A vertex cover in
$G_{\BB}^*$ extended with all the vertices incident
with a loop in $G_{\BB}$ forms a vertex cover in $G_{\BB}$.
Hence $\nu(G_{\BB}) \geq \tau(G_{\BB})$.
\end{proof}

In the analysis of the toughness of the chordal graph $G$,
we shall use the following technical lemma on trees:

\begin{lemma} \label{l:tree} 
Let $T$ be a tree.
For $i\in\Setx{0,1,2}$, let $E_i \subseteq E(T)$ be
such that every edge of $E_i$ is incident with exactly $i$
vertices of degree at most $2$.
For every $\frac{1}{3} \leq k \leq \frac{1}{2}$, the graph
$T - (E_0 \cup E_1 \cup E_2)$ has at least
$1 + k|E_{0}| + (1-k)|E_{1}| + |E_{2}|$ components that contain
a vertex whose degree in $T$ is at most $2$.
\end{lemma}

\begin{proof}
Let $E_* = E_0 \cup E_1 \cup E_2$.
For a tree $T$ and a subset $E$ of its edge set, let $c_2(T - E)$
denote the number of components of the forest $T - E$ that contain
a vertex whose degree in the tree $T$ is at most $2$.

We proceed by induction on the number of vertices of degree $2$.
Suppose $T$ contains no such vertex. (Thus, the only vertices of degree at most $2$ are the leaves of $T$.)
If $|E_2| \geq 1$, then $T$ is a tree on $2$ vertices and the statement holds,
so in addition we can suppose $|E_2| = 0$. 
We consider all components of $T - E_*$ that contain a leaf of $T$.
For each such component, we contract all edges in the subtree of $T$ that corresponds to this component,
and if the resulting vertex is not a leaf, then we add a new leaf adjacent to this vertex;
we let $T'$ denote the resulting tree.
We let $\ell$ denote the number of leaves of $T'$.
Since $T$ contains no vertex of degree $2$, we have $\ell = c_2(T - E_*)$.
Furthermore, $T'$ contains no vertex of degree $2$.
By an easy inductive argument, such a tree has at most
$2\ell-2$ vertices, and therefore at most $2\ell-3$ edges.
In conjunction with $|E_2| = 0$, this implies the following bound:
\begin{equation}\label{eq:1ell}
2\ell - 3 \geq |E_0| + |E_1|.
\end{equation}

The absence of degree $2$ vertices in $T$ implies that
every edge of $E_1$ is incident with a leaf in $T$, which yields
\begin{equation}\label{eq:2ell}
\ell \geq |E_1|.
\end{equation}

To show that $c_2(T-E_*) \geq 1 + k|E_{0}| + (1-k)|E_{1}|$,
we consider the right hand side of this inequality in the form
$1 + k(|E_{0}| + |E_{1}|) + (1-2k)|E_{1}|$.
By~\eqref{eq:1ell} and~\eqref{eq:2ell}, we have for $\frac{1}{3} \leq k \leq \frac{1}{2}$,
\begin{align*}
1 + k(|E_{0}| + |E_{1}|) + (1-2k)|E_{1}| 
&\leq 1 + k(2\ell - 3) + (1-2k)\ell \\
&= 1 - 3k + \ell \leq \ell = c_2(T-E_*).
\end{align*}

Thus, the lemma holds for a tree that contains no vertex of degree $2$.

Suppose that $T$ contains a vertex $u$ of degree $2$.
We let $T_1$ and $T_2$ be the two subtrees of $T$
such that $u$ is the only common vertex of $T_1$ and $T_2$
and every vertex of $T$ is in $T_1$ or $T_2$.
We observe that for $i\in\Setx{0,1,2}$ and $j\in\Setx{1,2}$,
every edge in $E^j_i = E_i \cap E(T_j)$ is incident with exactly $i$
vertices of degree at most $2$ in $T_j$,
so by induction the statement holds for 
$T_j$ with the sets of edges $E^j_i$ playing the role of $E_i$.
The trees $T_1$ and $T_2$ have no common edge, so
$|E_i| = |E^1_i| + |E^2_i|$ for $i\in\Setx{0,1,2}$, and we have:
\begin{align*}
c_2(T  - E_*)
&= c_2(T_1  - E_*) + c_2(T_2  - E_*) - 1 \\
&\geq 1 + k|E^1_0| + (1-k)|E^1_1| + |E^1_2| +
1 + k|E^2_0| + (1-k)|E^2_1| + |E^2_2| -1 \\
&= 1 + k|E_{0}| + (1-k)|E_{1}| + |E_{2}|.
\end{align*}
\end{proof}

In relation to an edge $e$ of $T$, we say that two vertices $u$, $v$
of $G$ form an \emph{$e$-enclosing pair} if there is a pair of
substantial vertices
$s \in V(F_u)$ and $t \in V(F_v)$ such that $s$ and $t$ are in different
components of $T - e$.

Lemma~\ref{l:disc} is a key part of the argument
relating vertex covers of $G_{\BB}$ to disconnecting sets of $G$.

\begin{lemma} \label{l:disc}
Let $e$ be an edge of $T$ and let $A_e$ be the overspan graph
assigned to $e$.
Let $C$ be a vertex cover of $A_e$.
We define a set $S \subseteq V(G)$ as follows:
In case $e$ is a black edge, let $S = C$,
or in case $e$ is a red edge, let $x$ be the corresponding vertex of $I$
and let $S = C \cup \{ x \}$.
If vertices $u$, $v$ of $G - S$ form an $e$-enclosing pair,
then $u$ and $v$ are in different components of $G - S$.
\end{lemma}

\begin{proof}
We first claim that $G - S$ consists of vertices whose corresponding
subtree in $\FF$ contains substantial vertices from exactly one component
of $T - e$.
Let $r, s$ be the (substantial) vertices incident with the edge $e$ in $T$.
Let $w$ be a vertex of $G$ such that $F_w$ contains a substantial vertex
from each component of $T - e$. Hence $F_w$ contains $r$ and $s$.
We show that $w \in S$.
If $w \in I$, then $e$ is a red edge and $w = x$. If $w \in V(G) \sm I$, then by 
the construction of $A_e$, $w$ is incident with a loop in $A_e$,
hence $w \in C$.
For every vertex of $G - S$, the corresponding subtree in $\FF$ does not
contain a vertex from each component of $T - e$.
The claim follows from Proposition~\ref{p:subst}.
Moreover, observe that if two vertices are adjacent in $G - S$, then
the two corresponding subtrees in $\FF$ contain vertices from the same
component of $T - e$.

Let $u$, $v$ be vertices of $G - S$ that form an $e$-enclosing pair.
So $F_u$ contains vertices from one component of $T - e$ and $F_v$
contains vertices from the other component. 
In particular, $u \neq v$.
Let $U$ be the set of all vertices of $G - S$ such that the corresponding
subtrees in $\FF$ contain substantial vertices from the same
component of $T - e$ as the subtree $F_u$, and let $V$ be the set
of vertices of $G - S$ that are not in $U$.  
We conclude that there is no edge from $U$ to $V$ in $G - S$,
hence there is no path from $u$ to $v$. The vertices $u$ and $v$ are in
different components of $G - S$.
\end{proof}
 
We are now ready to prove Theorem~\ref{t:main},
showing that every $10$-tough chordal graph on at least $3$ vertices is Hamiltonian.

\begin{proof}[Proof of Theorem~\ref{t:main}]
Let $G$ be a $10$-tough chordal graph on at least $3$ vertices,
and for the sake of the contradiction suppose that $G$ is not Hamiltonian.
By Lemma~\ref{l:match ham}, there is a subfamily $\BB_0 \subseteq \AAA$
such that $\nu(G_{\BB_0}) \leq 2\size{\BB_0}-2$ and
by Lemma~\ref{l:konig}, we also have $\tau(G_{\BB_0}) \leq 2\size{\BB_0}-2$.
Let $C$ be a minimum vertex cover of $G_{\BB_0}$;
we fix $C$ and extend $\BB_0$ to a maximal subfamily $\BB$ such that $C$ is
a vertex cover of $G_\BB$.
Clearly, $\size C \leq 2\size\BB-2$.
We produce a separating set $S\subseteq V(G)$
demonstrating that $G$ is not $10$-tough; to find it, we augment $C$
as follows.

Let $B$ be the set of all $\BB$-edges of $T$. Let $E'$ be the set of all
red edges of $B$ such that none of the adjacent (black) edges of $T$
belongs to $B$. Every red edge
$e$ corresponds to an $I$-path, say $F_{v_e}$; let $X'$ be the set of
all vertices $v_e$ of $G$ such that $e\in E'$. We set $S = C\cup X'$
and show that it has the required properties.

Let $E_*$ be the set of all black edges that belong to $B$. For
$i\in\Setx{0,1,2}$, let $E_i \subseteq E_*$ consist of edges incident
with exactly $i$ vertices whose degree in $T$ is at most~$2$.
Clearly, $\size{E_*} = \size{E_0} + \size{E_1} + \size{E_2}$.
A black edge of $E_i$ is adjacent to at most
$i$ red edges, and every red edge in $B \sm E'$ is adjacent to
a black edge of $B$, hence 
$\size{B \sm (E_* \cup E')} \leq \size{E_1} + 2\size{E_2}$.
By the definition of $\AAA$, there are two overspan graphs assigned
to every black edge, hence
$\size{\BB} \leq 2\size{E_0} + 3\size{E_1} + 4\size{E_2} + \size{E'}$.
We bound the size of the separating set $S$:
\begin{equation}\label{eq:1}
  \size{S} = \size{C} + \size{X'} \leq 2\size\BB-2 + \size{E'}
  < 4\size{E_0} + 6\size{E_1} + 8\size{E_2} + 3\size{E'}.
\end{equation}

In order to bound the number of components $c(G - S)$, let us
start with $c(G - C)$.
Observe that for every substantial vertex of degree at most $2$,
there is an $I$-path that contains this vertex. Furthermore,
every trivial $I$-path contains exactly one substantial vertex and
every non-trivial $I$-path contains exactly two substantial vertices
that are connected by a red edge in $T$.
Note that $T$ with the sets of edges $E_0, E_1, E_2$ fit the
criteria of Lemma~\ref{l:tree}, which we apply with $k = \frac{2}{5}$.
Consequently, the graph $T - E_*$ has more than
$\frac{2}{5}|E_{0}| + \frac{3}{5}|E_{1}| + |E_{2}|$ 
components that contain a vertex whose degree in $T$ is at most $2$.
Associate one vertex $v$ of $I$ with each of these components such that
the component contains substantial vertices of $F_v$.
For any pair of these associated vertices, there is an edge $e$ of $E_*$
such that the vertices form an $e$-enclosing pair. By Lemma~\ref{l:disc}
these vertices are in different components of $G - C$. We obtain
$c(G - C) > \frac25\size{E_0} + \frac35\size{E_1} + \size{E_2}$.

We continue by bounding $c(G - S)$.
For every vertex $v_e \in X'$, there is a corresponding edge 
$e \in E'$ and the overspan graph $A_e$.
Let $d$, $d'$ denote the edges adjacent to $e$ in $T$. 
Let us consider the graph $A_d$.
(The argument for $A_{d'}$ is symmetric.)
By the definition of $E'$, we have $A_d \not \in \BB$.
Due to the maximality of $\BB$ the set $C$ is not
a vertex cover of the graph $G_{\BB \cup \{A_d\} }$.
Thus, the graph $A_d$ contains an edge $e_0$ (a simple edge or a loop)
such that no vertex incident with this edge is in $C$.
In $T$, the edges $d$ and $e$ have a common substantial vertex, say $t$.
Choose a vertex $u$ of $G$ such that $t \in V(F_u)$ and $u$ is incident
with $e_0$ in $A_d$.
Since $t \in V(F_{v_e})$, the vertices $u$ and $v_e$ are adjacent in $G$.
Observe that $u \not \in C \cup I$. 
Similarly, there is a substantial vertex $t'$ and a vertex
$u' \in V(G)\sm (C \cup I)$ such that
$t' \in V(F_{u'})$ and $t' \in V(F_{v_e})$. The vertices $u$ and $u'$ form
an $e$-enclosing pair.
The three vertices $u, v_e, u'$ are in the same component of
the graph $G - C$. 
By Lemma~\ref{l:disc}, removing the vertex $v_e$ disconnects this
component into two components such that $u$ is in one of them 
and $u'$ is in the other.  
Removing the vertices of $X'$ from $G - C$
increases the number of components by $\size{X'}$.
Therefore we obtain: 
\begin{equation}\label{eq:2}
  c(G - S) > \frac25\size{E_0} + \frac35\size{E_1} + \size{E_2} + \size{E'}.
\end{equation}

Comparing~\eqref{eq:1} and \eqref{eq:2}, we find that $G$
is not $10$-tough. We obtain a contradiction proving 
Theorem~\ref{t:main}.
\end{proof}

We remark that the bound of Theorem~\ref{t:main} is
still far from the lower bound of `almost' $\frac74$ proven
in~\cite{74}, and there seems to be ample room for further
improvements.


\section{Toughness and Hamilton-connectedness}
\label{sec:hc}

With a little extra work, one can use the method of this paper to
obtain a slightly stronger result than Theorem~\ref{t:main}, namely
that any $10$-tough chordal graph $G$ is
Hamilton-connected. (Recall that this means that for any two vertices
$u,v$ of $G$, there is a Hamilton path from $u$ to $v$.)

Assume that the vertices $u$ and $v$ are given. Let us sketch the main
modifications required to show that $G$ admits a Hamilton path from
$u$ to $v$:
\begin{itemize}
\item in Lemma~\ref{l:ham}, we additionally assume that the matching
  chosen from the graphs in $\AAA(G)$ is incident with neither $u$ nor
  $v$, 
\item in the proof of Lemma~\ref{l:ham}, the Euler tour is replaced by
  a trail from a vertex of $F_u$ to a vertex of $F_v$ spanning all the
  vertices of $T$, 
\item to find a matching as above, it is sufficient to increase the bound
  on the matching number of the graph $G_\BB$ in Lemma~\ref{l:match
    ham} by two, to $2\size{\BB}$.
\end{itemize}
By inspecting inequality~\eqref{eq:1}, one can see that the proof of
Theorem~\ref{t:main} works just the same even with the strengthened
assumption in Lemma~\ref{l:match ham}. 

We hope that the interested reader will be able to reconstruct the
argument from this account.

\section*{Acknowledgement}
We thank the anonymous referees for their helpful remarks and suggestions.


\let\OLDthebibliography\thebibliography
\renewcommand\thebibliography[1]{
  \OLDthebibliography{#1}
  \setlength{\parskip}{0pt}
  \setlength{\itemsep}{4pt plus 0.3ex}
}

\end{document}